\documentclass[ 12pt,a4paper,openany]{article}

\usepackage{amsmath,amssymb,amsfonts,amsthm}
\usepackage{graphics}
\usepackage{mathrsfs}
\usepackage[all]{xy}

\newtheorem{thm}{Theorem}[section]
\newtheorem{prop}[thm]{Proposition}
\newtheorem{lemma}[thm]{Lemma}

\newtheorem{claim}{Claim}
\theoremstyle{definition}

\newtheorem{rem}[thm]{Remark}

\newtheorem*{ack}{Acknowledgment}

\newcommand{\sq}{\hfill $\square$}
\newcommand{\kah}{K\"{a}hler}
\newcommand{\dol}{\sqrt{-1}\partial \overline{\partial}}

\newcommand{\hol}{H\"{o}lder}
\newcommand{\ma}{Monge-Amp$\grave{{\rm e}}$re}

\makeatletter
\def\address#1#2{\begingroup
\noindent\parbox[t]{16cm}{%
\small{\scshape\ignorespaces#1}\par\vskip1ex
\noindent\small{\itshape E-mail address}%
\/: #2\par\vskip4ex}\hfill%
\endgroup}%
\makeatother

\makeatletter

\@addtoreset{equation}{section}
\makeatother

\title{Almost scalar-flat K\"{a}hler metrics on affine algebraic manifolds}
\author{Takahiro Aoi}
\date{}
\begin{document}
\maketitle

\footnote{ %2010 MSC numbers
}
\footnote{ %key words and phrases
}

\begin{abstract}
Let $(X,L_{X})$ be an $n$-dimensional polarized manifold.\
Let $D$ be a smooth hypersurface defined by a holomorphic section of $L_{X}$.\
In this paper, we show the existence of a complete K\"{a}hler metric on $X \setminus D$ whose scalar curvature is flat away from some divisor if there are positive integers $l(>n),m$ such that the line bundle $K_{X}^{-l} \otimes L_{X}^{m}$ is very ample and the ratio $m/l$ is sufficiently small.\
\end{abstract}

\setcounter{tocdepth}{2}
\tableofcontents

\section{
Introduction}
\label{intro}
%The existence of constant scalar curvature {\kah} (cscK) metrics on complex manifolds is a fundamental problem in {\kah} geometry.\
%If a complex manifold is noncompact, there are many positive results in this problem.\
%In 1979, Calabi \cite{Ca} showed that if a Fano manifold has a {\kah} Einstein metric, then there is a complete Ricci-flat {\kah} metric on the total space of the canonical line bundle.\
%In addition, there exist following generalizations.\
%In 1990, Bando-Kobayashi \cite{BK} showed that if a Fano manifold admits an anti-canonical smooth divisor which has a Ricci-positive {\kah} Einstein metric, then there exists a complete Ricci-flat {\kah} metric on the complement.\
%In 1991, Tian-Yau \cite{TY} showed that if a Fano manifold admits an anti-canonical smooth divisor which has a Ricci-flat {\kah} metric, then there is a complete Ricci-flat {\kah} metric on the complement.\
%In 2002, on the other hand, as a scalar curvature version of Calabi's result \cite{Ca}, Hwang-Singer \cite{HS} showed that if a polarized manifold has a nonnegative cscK metric, then the total space of the dual line bundle admits a complete scalar-flat {\kah} metric.\
%However, a similar generalization of Hwang-Singer \cite{HS} like Bando-Kobayashi \cite{BK} and Tian-Yau \cite{TY} is unknown since it is hard to solve a forth order nonlinear partial differential equation.\

%The setting of this paper is as follows.
Let $(X,L_{X})$ be a polarized manifold of dimension $n$, i.e., $X$ is an $n$-dimensional compact complex manifold and $L_{X}$ is an ample line bundle over $X$.\
Assume that there is a smooth hypersurface $D \subset X$ with
$$
D \in |L_{X}|.
$$
Set an ample line bundle $L_{D} := \mathscr{O}(D)|_{D} = L_{X}|_{D}$ over $D$.\
Since $L_{X}$ is ample, there exists a Hermitian metric $h_{X}$ on $L_{X}$ which defines a {\kah} metric $\theta_{X}$ on $X$, i.e., the curvature form of $h_{X}$ multiplied by $\sqrt{-1}$ is positive definite.\
Then, the restriction of $h_{X}$ to $L_{D}$ defines also a {\kah} metric $\theta_{D}$ on $D$.\ 
Let $\hat{S}_{D}$ be the average of the scalar curvature $S(\theta_{D})$ of $\theta_{D}$ defined by
$$
\hat{S}_{D} := \frac{\displaystyle\int_{D} S(\theta_{D}) \theta_{D}^{n-1}}{\displaystyle\int_{D} \theta_{D}^{n-1}} = \frac{(n-1) c_{1}(K_{D}^{-1}) \cup c_{1}(L_{D})^{n-2}}{c_{1}(L_{D})^{n-1}},
$$
where $K_{D}^{-1}$ is the anti-canonical line bundle of $D$.\
Note that $\hat{S}_{D}$ is a topological invariant in the sense that it is representable in terms of Chern classes of the line bundles $K_{D}^{-1}$ and $L_{D}$.\
In this paper, we treat the following case :
\begin{equation}
\label{positivity 1}
\hat{S}_{D} > 0.
\end{equation}
Let $\sigma_{D} \in H^{0}(X,L_{X})$ be a defining section of $D$ and set $t := \log || \sigma_{D} ||_{h_{X}}^{-2}$.\
Following \cite{BK}, we can define a complete {\kah} metric $\omega_{0}$ by
$$
\omega_{0}:= \frac{n(n-1)}{\hat{S}_{D}}\dol \exp \left(\frac{\hat{S}_{D}}{n(n-1)}t\right)\\
$$
on the noncompact complex manifold $X \setminus D$.\
This {\kah} metric $\omega_0$ is of asymptotically conical geometry (see \cite{Aoi1}).

In \cite{Aoi1}, we show that there exists a complete scalar-flat {\kah} metric which is of asymptotically conical geometry if the following conditions hold :
(1) $ n\geq 3$ and there is no nonzero holomorphic vector field on $X$ vanishing on $D$,
(2) $\theta_{D}$ is a cscK metric and $0 < \hat{S}_D < n(n-1)$,
(3) the scalar curvature of $\omega_0$ is sufficiently small in the weighted Banach space (see Condition 1.2 and Condition 1.3 in \cite{Aoi1}).
In this paper, we construct a complete {\kah} metric on $X \setminus D$ whose scalar curvature can be made small arbitrarily by gluing plurisubharmonic functions.\

To show this, we consider a degenerate (meromorphic) complex {\ma} equation.\
Take positive integers $l > n$ and $m$ such that the line bundle $K_{X}^{-l} \otimes L_{X}^{m}$ is very ample.\
Let $F \in |K_{X}^{-l} \otimes L_{X}^{m}|$ be a smooth hypersurface defined by a holomorphic section $\sigma_{F} \in H^{0}(X,K_{X}^{-l} \otimes L_{X}^{m})$ such that the divisor $D+F$ is simple normal crossing.\
For a defining section $\sigma_{D} \in H^{0}(X,L_{X})$ of $D$, set
$$
\xi := \sigma_{F} \otimes \sigma_{D}^{-m}.
$$
From the result due to Yau \cite[Theorem 7]{Yau}, we can solve the following degenerate complex Monge-Amp$\grave{{\rm e}}$re equation:
$$
(\theta_{X} + \dol \varphi)^{n} = \xi^{-1/l} \wedge \overline{\xi}^{-1/l}.
$$
Moreover, it follows from a priori estimate due to  Ko\l odziej \cite{Ko} that the solution $\varphi$ is bounded on $X$.\
Thus, we can glue plurisubharmonic functions by using the regularized maximum function.\
To compute the scalar curvature of the glued {\kah} metric, we need to study behaviors of higher order derivatives of the solution $\varphi$.\
So, we give explicit estimates of them near the intersection $D \cap F$ :
\begin{thm}
\label{solution}
Let $(z^{i})_{i=1}^{n} = (z^{1},z^{2},...,z^{n-2},w_{F},w_{D})$ be local holomorphic coordinates such that $\{ w_{F} = 0 \} = F$ and $\{ w_{D} = 0 \} = D$.\
Then, there exists a positive integer $a(n)$ depending only on the dimension $n$ such that 

\begin{eqnarray*}
\left| \frac{\partial^{2}}{\partial z^{i} \partial \overline{z}^{j}} \partial^{\alpha} \varphi \right| &=& O\left(  |w_{D}|^{-2a(n) m/l}|w_{F}|^{-2a(n) / l} \right),\\
\left| \frac{\partial^{4}}{\partial w_{F}^{2} \partial \overline{w_{F}^{2}}}  \varphi \right| &=& O\left( |w_{D}|^{-2 a(n) m/l}|w_{F}|^{ -2 - 2a(n)/l } \right),\\
\left| \frac{\partial^{4}}{\partial w_{D}^{2} \partial \overline{w_{D}^{2}}} \varphi \right| &=& O\left( |w_{D}|^{ -2 -2a(n)m/l }|w_{F}|^{-2a(n)/l} \right),
\end{eqnarray*}
as $|w_{F}|,|w_{D}| \to 0$, for any $1\leq i, j \leq n-2$ and multi-index $\alpha = (\alpha_{1},...,\alpha_{n})$ with $0 \leq \sum_{i} \alpha_{i} \leq 2$.\

\end{thm}

By applying Theorem \ref{solution} and gluing plurisubharmonic functions, we have the following result :
\begin{thm}
\label{small scalar curvature}
Assume that there exist positive integers $l>n$ and $m$ such that
\begin{equation}
\label{positivity 2}
\frac{a(n)m}{2l} < \frac{\hat{S}_{D}}{n(n-1)}
\end{equation}
and the line bundle $K_{X}^{-l} \otimes L_{X}^{m}$ is very ample.\
Here, $a(n)$ is the positive integer in Theorem \ref{solution}.\
Take a smooth hypersuface $F \in |K_{X}^{-l} \otimes L_{X}^{m}|$ such that $D + F$ is simple normal crossing.\
Then, for any relatively compact domain $Y \Subset X \setminus (D \cup F)$, there exists a complete {\kah} metric $\omega_{F}$ on $X \setminus D$ whose scalar curvature $S(\omega_{F}) = 0$ on $Y$ and is arbitrarily small on the complement of $Y$.\
In addition, $\omega_{F} = \omega_{0}$ on some neighborhood of $D \setminus (D \cap F)$.\
\end{thm}

For example, if the anti-canonical line bundle $K_{X}^{-1}$ of the compact complex manifold $X$ is nef (in particular, $X$ is Fano), the assumption (\ref{positivity 2}) in Theorem \ref{small scalar curvature} holds, i.e., we can always find such integers $l,m$.\
In this article, we treat the case that $K_{X}^{-1}$ has positivity in the senses of (\ref{positivity 1}) and (\ref{positivity 2}).\
From \cite{Aoi1}, if there exists a complete {\kah} metric which is of asymptotically conical geometry and satisfies Condition 1.2 and Condition 1.3, $X \setminus D$ admits a complete scalar-flat {\kah} metric.\
In fact, Theorem \ref{small scalar curvature} gives a {\kah} metric whose scalar curvature is under control.\
However, the {\kah} metric $\omega_{F}$ in Theorem \ref{small scalar curvature} is not of asymptotically conical geometry (near the intersection of $D$ and $F$).\
This problem will be solved in \cite{Aoi3}.

This paper is organized as follows.\
In Section 2, we construct {\kah} potentials, i.e., strictly plurisubharmonic functions, whose scalar curvature is under control.\
In addition, we glue these plurisubharmonic functions by using the regularized maximum function.\
In Section 3, we prove Theorem \ref{solution}.\
To show this, we recall the $C^{2,\epsilon}$-estimate of a solution of the degenerate complex {\ma} equation.\
In Section 4, we prove Theorem \ref{small scalar curvature}.\

\begin{ack}
The author would like to thank Professor Ryoichi Kobayashi who first brought the problem in this article to his attention, for many helpful comments.\
\end{ack}

\section{
Plurisubharmonic functions with small scalar curvature}
\label{sec:7}
To prove Theorem \ref{small scalar curvature}, we prepare {\kah} potentials, i.e., strictly plurisubharmonic functions, whose scalar curvature is under control.\
\subsection{
{\kah} potential near $D$}

In this subsection, we consider a {\kah} potential near $D$ and study the scalar curvature of it.\
Recall that
\begin{equation}
\label{herm}
t = \log || \sigma_{D} ||^{-2}
\end{equation}
and $\theta_{X} = \dol t = \dol \log || \sigma_{D} ||^{-2}$ on $X \setminus D$.\
Set
\begin{equation}
\label{potential D}
\Theta(t) = \frac{n(n-1)}{\hat{S}_{D}} \exp \left( \frac{\hat{S}_{D}}{n(n-1)}t \right).
\end{equation}
Following \cite{BK}, we can define a complete {\kah} metric by
$$
\omega_{0}:= \dol \Theta(t) = \frac{n(n-1)}{\hat{S}_{D}}\dol \exp \left(\frac{\hat{S}_{D}}{n(n-1)}t\right)\\
$$
on $X \setminus D$.\
Following \cite{Aoi1}, recall the asymptotic behavior of the scalar curvature of $\omega_{0}$.\
\begin{lemma}
The scalar curvature $S(\omega_{0})$ can be estimated as follows :
$$
\label{estimate zero}
S(\omega_{0}) = O \left( || \sigma_{D} ||^{2\hat{S}_{D}/n(n-1)} \right)
$$
as $\sigma_{D} \to 0$.\
\end{lemma}

\begin{rem}
Moreover, from Theorem 1.1 in \cite{Aoi1}, if $\theta_{D}$ is cscK, we have the following strong result :
$$
S(\omega_{0}) = O \left( || \sigma_{D} ||^{2 + 2\hat{S}_{D}/n(n-1)} \right)
$$
as $\sigma_{D} \to 0$.\
\end{rem}

\subsection{
{\kah} potential near $F$}

In this subsection, we construct a {\kah} metric on $X$ whose scalar curvature is small near the smooth hypersurface $F \in |K_{X}^{-l} \otimes L_{X}^{m}|$.\
Here, $l,m$ are positive integers such that the line bundle $K_{X}^{-l} \otimes L_{X}^{m}$ is very ample.\
For a fixed Hermitian metric on $K_{X}^{-l} \otimes L_{X}^{m}$, set $b := \log || \sigma_{F} ||^{-2}$.\
Since the holomorphic line bundle $K_{X}^{-l} \otimes L_{X}^{m}$ is very ample, we may assume that $\dol b$ is a {\kah} metric on $X$.\
For parameters $v > 0$ and $\beta \in \mathbb{Z}_{>0}$, define a function by
\begin{equation}
\label{function g}
G_{v}^{\beta}(b) := \int_{b_{0}}^{b} \left( \frac{1}{e^{-y} + v} \right)^{1/\beta} dy
\end{equation}
for some fixed $b_{0} \in \mathbb{R}$.\
Note that $G_{v}^{\beta}(b)$ is defined smoothly outside $F$ and $\lim_{b \to \infty}G_{v}^{\beta}(b) = + \infty$ for any $v > 0$.\

\begin{lemma}
\label{gamma}
For $\mathbb{Z} \ni \beta \geq 1$, $\gamma_{v}^{\beta} := \dol G_{v}^{\beta}(\beta b)$ defines a {\kah} metric on $X$.\
\end{lemma}

\begin{proof}
In fact,
\begin{eqnarray*}
\dol G_{v}^{\beta}(\beta b)
&=& \beta \sqrt{-1} \partial \left[ \left( \frac{1}{e^{- \beta b} + v} \right)^{1/\beta} \overline{\partial} b \right] \\
&=& \left( \frac{1}{e^{- \beta b} + v} \right)^{1/\beta} \left( \beta \dol b + \frac{ e^{-\beta b}}{e^{-\beta b} + v} \sqrt{-1} \partial b \wedge \overline{\partial} b \right).
\end{eqnarray*}
Note that the last term
$$
\frac{e^{-\beta b}}{e^{-\beta b} + v} \sqrt{-1} \partial b \wedge \overline{\partial} b
$$
is defined smoothly on $X$ from the assumption that $\mathbb{Z} \ni \beta \geq 1$.\
Since $\dol b$ is a {\kah} metric on $X$, we finish the proof.
\end{proof}

Next,  the scalar curvature of $\gamma_{v}^{\beta}$ is given by

\begin{lemma}
\label{small F}
For $\beta \geq 3$, we obtain
$$
S(\gamma_{v}^{\beta}) = S(\dol G_{v}^{\beta}(\beta b)) = O ((|| \sigma_{F} ||^{2 \beta} + v)^{1/\beta})
$$
as $|| \sigma_{F} || \to 0$.\
\end{lemma}

\begin{proof}
This lemma follows from the similar way in the computation of the scalar curvature of $\omega_{0}$.\
In fact, since
\begin{eqnarray*}
\left( (\dol G_{v}^{\beta}(\beta b) \right)^{n}
&=& \beta^{n} \left( \frac{1}{e^{- \beta b} + v} \right)^{n / \beta} \left( 1 + \frac{  e^{- \beta b}}{\beta(e^{-\beta b} + v)} ||\partial b||^{2} \right) (\dol b)^{n},
\end{eqnarray*}
we have
\begin{eqnarray*}
{\rm Ric}(\dol G_{v}^{\beta}(\beta b))
&=& {\rm Ric}(\dol b) - \dol \log \left( 1 + \frac{ e^{- \beta b}}{\beta(e^{-\beta b} + v)} ||\partial b||^{2} \right) \\
&\hspace{19pt}+& \frac{n}{ \beta} \left( \frac{1}{e^{-\beta b}+v} \dol e^{-\beta b} + \frac{\beta }{ (e^{-\beta b}+v)^{2}} \sqrt{-1} \partial e^{-\beta b} \wedge \overline{\partial} e^{-\beta b} \right).
\end{eqnarray*}
Note that second and last terms above are zero on $F$.\
Thus, when we consider the scalar curvature $S(\gamma_{v}^{\beta})$, it is enough to see the term $1/(e^{-\beta b} + v)^{1/\beta} \dol b$ and the Ricci form ${\rm Ric}(\dol b)$.\
Therefore the desired result is obtained.
\end{proof}

\begin{rem}
\label{small F2}
If the value of the function $e^{- \beta b} = || \sigma_{F} ||^{2 \beta}$ is compatible with $v$, i.e., $|| \sigma_{F} ||^{2 \beta} \approx v$, we have the following estimate of $S(\dol G_{v}^{\beta}(\beta b))$ :
$$
S(\dol G_{v}^{\beta}(\beta b)) = O(1).
$$
However, we will consider the case that $|| \sigma_{F} ||^{2 \beta} \approx v^{k }$ for sufficiently large $k \in \mathbb{N}$ which will be specified later.\
Namely, it suffices to consider a {\it sufficiently} small neighborhood of $F$ defined by the inequality $|| \sigma_{F} ||^{2 \beta} \leq v^{k }$ and Lemma \ref{small F} holds on this region.\
\end{rem}

\subsection{
Ricci-flat {\kah} metric away from $D \cup F$}

In this subsection, we study an incomplete Ricci-flat {\kah} metric away from the support of the divisor $D+F$.\
Recall the setting in Theorem \ref{small scalar curvature}.\
Let $l>n$ and $m$ be positive integers such that there exists a holomorphic section $\sigma_{F} \in H^{0}(K_{X}^{-l} \otimes L_{X}^{m})$ which defines a smooth hypersurface $F \subset X$, i.e., $(\sigma_{F})_{0} = F$.\
It follows from the hypothesis of the average value $\hat{S}_{D}$ of the scalar curvature that divisors $D$ and $F$ intersect to each other.\
Set
$$
\xi := \sigma_{F} \otimes \sigma_{D}^{-m}.\
$$
Note that $\xi$ is a meromorphic section of $K_{X}^{-l}$.\
Then, define a singular and degenerate volume form $V$ by
$$
V := \xi^{-1/l} \wedge \overline{\xi^{-1/l}}.
$$
From the construction above, $V$ has finite volume on $X$ and its curvature form, i.e., the Ricci form, is zero on the complement of $D \cup F$.\
For the {\kah} metric $\theta_{X}$ on $X$, write
$$
V = f \theta_{X}^{n}
$$
for some non-negative function $f$ on $X$ with the normalized condition
$$
\int_{X} V = \int_{X}  f \theta_{X}^{n} = \int_{X} \theta_{X}^{n}.
$$
We know that $f$ is smooth away from $D \cup F$.\
From the result due to Yau \cite[Theorem 7]{Yau}, recall the solvability of a meromorphic complex Monge-Amp$\grave{{\rm e}}$re equation :
\begin{thm}
Let $L_{1}$ and $L_{2}$ be holomorphic line bundles over a compact {\kah} manifold $(X,\theta_{X})$.\
Let $s_{1}$ and $s_{2}$ be nonzero holomorphic sections of $L_{1}$ and $L_{2}$, respectively.\
Let $F$ be a smooth function on $X$ such that $\int_{X} |s_{1}|^{2k_{1}} |s_{2}|^{-2k_{2}} \exp (F) \theta_{X}^{n} = {\rm Vol} (X)$, where $k_{1} \geq 0$ and $k_{2} \geq 0$.\
Suppose that $\int_{X} |s_{2}|^{-2nk_{2}} < \infty$ for $n = \dim X$.\
Then, we can solve the following equation
$$
(\theta_{X} + \dol \varphi)^{n} = |s_{1}|^{2k_{1}} |s_{2}|^{-2k_{2}} \exp (F) \theta_{X}^{n}
$$
so that $\varphi$ is smooth outside divisors of $s_{1}$ and $s_{2}$ with $\sup_{X} \varphi < + \infty$.\
\end{thm}
Then, we can solve the following complex Monge-Amp$\grave{{\rm e}}$re equation
\begin{equation}
\label{solution by Yau}
(\theta_{X} + \dol \varphi)^{n} = f \theta_{X}^{n} = \xi^{-1/l} \wedge \overline{\xi^{-1/l}}.
\end{equation}
with $\varphi \in C^{\infty}(X \setminus D \cup F)$.\
Thus, we obtain a Ricci-flat {\kah} metric $\theta_{X} + \dol \varphi $ on the complement of $D \cup F$.\
For this solution $\varphi$, we obtain the following a priori estimate due to Ko\l odziej \cite{Ko} (see also \cite{GZ}):
\begin{thm}
If $f$ is in $L^{p}(\theta_{X}^{n})$ for some $p>1$, we have
$$
{\rm Osc}_{X} \varphi \leq C
$$
for some $C>0$ depending only on $\theta_{X}$ and $||f||_{L^{p}}$.\
\end{thm}

\subsection{
Gluing plurisubharmonic functions}

In this subsection, following \cite[Chapter I]{De}, we consider gluing {\kah} potentials, i.e., plurisubharmonic functions, obtained in previous subsections.\
Let $\rho \in C^{\infty}(\mathbb{R},\mathbb{R})$ be a nonnegative function with support in $[-1,1]$ such that $\int_{\mathbb{R}} \rho (h) dh = 1$ and $\int_{\mathbb{R}} h \rho (h) dh = 0$.\

\begin{lemma}[the regularized maximum]
\label{gluing lemma}
For arbitrary $\eta = (\eta_{1},...,\eta_{p}) \in (0,+\infty)^{p}$, the function
$$
M_{\eta}(t_{1},...,t_{p}) = \int_{\mathbb{R}^{p}} \max \{ t_{1} + h_{1},...,t_{p} + h_{p} \} \prod_{1 \leq j \leq p} \eta_{j}^{-1} \rho(h_{j}/\eta_{j}) dh_{1}...dh_{p}
$$
called the regularized maximum possesses the following properties $:$
\begin{description}
\item[a)] $M_{\eta}(t_{1},...,t_{p})$ is non decreasing in all variables, smooth and convex on $\mathbb{R}^{p} \hspace{2pt};$
\item[b)] $\max \{ t_{1},...,t_{p} \} \leq M_{\eta}(t_{1},...,t_{p}) \leq \max \{ t_{1} + \eta_{1},...,t_{p} + \eta_{p} \} \hspace{2pt};$
\item[c)] $M_{\eta}(t_{1},...,t_{p}) = M_{(\eta_{1},...,\hat{\eta_{j}},...,\eta_{p})}(t_{1},...,\hat{t_{j}},...,t_{p})$ if $t_{j}+\eta_{j} \leq \max_{k \neq j}\{ t_{k} - \eta_{k} \} \hspace{2pt} ;$
\item[d)] $M_{\eta}(t_{1}+a,...,t_{p}+a) = M_{\eta}(t_{1},...,t_{p})+a \hspace{2pt} ;$
\item[e)]  if $u_{1},...,u_{p}$ are plurisubharmonic and satisfy $H(u_{j})_{z}(\xi) \geq \gamma_{z}(\xi)$ where $z \mapsto \gamma_{z}$ is a continuous hermitian form on $TM$, then $u = M_{\eta}(u_{1},...,u_{p})$ is a plurisubharmonic and satisfies $Hu_{z}(\xi) \geq \gamma_{z}(\xi)$.\
\end{description}
\end{lemma}

\begin{rem}
Lemma \ref{gluing lemma} is a key in the proof of Richberg theorem (see \cite[p.43]{De}).\
In our case, we have already prepared three plurisubharmonic functions and must compute the Ricci form of the glued {\kah} metric later.\
Therefore, we need the explicit formula of the glued function.
\end{rem}
In addition, we obtain
\begin{lemma}
\label{M derivatives}
There exists a constant $C>0$ such that
$$
\left| \frac{\partial^{|\alpha|} M_{\eta}}{\partial t^{\alpha}} (t) \right| \leq C \min \{ \eta_{j} | \alpha_{j} \neq 0 \} \prod_{\alpha_{i} \neq 0} \eta_{i}^{-\alpha_{i}}
$$
for any multi index $\alpha = (\alpha_{i})_{i}$ with $1 \leq |\alpha| \leq 4$.\
\end{lemma}

Recall that the {\kah} potential of $\omega_{0}$ is given by
$$
\Theta(t) = \frac{n(n-1)}{\hat{S}_{D}} \exp \left( \frac{\hat{S}_{D}}{n(n-1)}t \right).
$$
For $\kappa \in (0,1)$, set
\begin{equation}
\label{potential F}
\tilde{G}_{v}^{\beta}(b) := G_{v}^{\beta}(\beta b) + \kappa \Theta(t).
\end{equation}
This constant $\kappa$ will be specified later.\
For this {\kah} potential, we have
\begin{lemma}
\label{estimate g}
For the complete {\kah} metric $\dol \tilde{G}_{v}^{\beta}(b)$ on $X \setminus D$, we have
\begin{eqnarray}
S(\dol(\tilde{G}_{v}^{\beta}(b))) = \left\{ \begin{array}{ll}
O(||\sigma_{D}||^{2\hat{S}_{D}/n(n-1)}) & \mbox{near $D$,}\\
O((||\sigma_{F}||^{2 \beta } + v)^{1 / \beta}) & \mbox{near $F$.}\\
\end{array} \right.
\end{eqnarray}
\end{lemma}

\begin{proof}
First, we study the behavior of the scalar curvature near $D$.\
Since
$$
|| \sigma_{D} ||^{2 + 2 \hat{S}/(n-1)} \left( \dol(\tilde{G}_{v}^{\beta}(b)) \right)^{n}
$$
is a smooth volume form on $X$, the Ricci form of $ \dol(\tilde{G}_{v}^{\beta}(b))$ given by
\begin{eqnarray*}
{\rm Ric} (\dol(\tilde{G}_{v}^{\beta}(b))
&=& - \left( \frac{\hat{S}}{n-1} + 1 \right) \theta_{X} - \dol \log || \sigma_{D} ||^{2 + 2 \hat{S}/(n-1)} \left( \dol(\tilde{G}_{v}^{\beta}(b)) \right)^{n}\\
\end{eqnarray*}
is defined smoothly on X.\
Recall that
$$
\dol(\tilde{G}_{v}^{\beta}(b)) = \kappa \omega_{0} + \gamma_{v}^{\beta}.
$$
As $\omega_{0}$ is of asymptotically conical geometry, we have the desired result near $D$.\
Similarly, the volume form
$$
(||\sigma_{F}||^{2 \beta } + v)^{n / \beta} \left( \dol(\tilde{G}_{v}^{\beta}(b)) \right)^{n}
$$
is smooth near $F \setminus (D \cap F)$.\
Then, the following identity
\begin{eqnarray*}
{\rm Ric} (\dol(\tilde{G}_{v}^{\beta}(b))
&=& \frac{n}{ \beta} \left( \frac{1}{e^{-\beta b}+v} \dol e^{-\beta b} + \frac{\beta }{ (e^{-\beta b}+v)^{2}} \sqrt{-1} \partial e^{-\beta b} \wedge \overline{\partial} e^{-\beta b} \right)\\
&\hspace{15pt}-& \dol \log (||\sigma_{F}||^{2 \beta } + v)^{n / \beta} \left( \dol(\tilde{G}_{v}^{\beta}(b)) \right)^{n}
\end{eqnarray*}
implies the desired result near $F$.
\end{proof}

In summary, we have prepared the three strictly plurisubharmonic functions $\Theta(t) = (n(n-1)/\hat{S}_{D})\exp ((\hat{S}_{D}/n(n-1))t), \tilde{G}_{v}^{\beta}(b) = G_{v}^{\beta}(\beta b) + \kappa \Theta(t), t + \varphi = \log || \sigma_{D} ||^{-2} + \varphi$ whose scalar curvature is under control.\
From Lemma \ref{gluing lemma}, we immediately have

\begin{prop}
For parameters $c, v, \eta$ and $\kappa \in (0,1)$, a function defined by
$$M_{c,v,\eta} := M_{\eta} \left( \Theta(t),\tilde{G}_{v}^{\beta}(b) ,t+\varphi+c \right)$$
is a strictly plurisubharmonic function on $X \setminus (D \cup F)$.\
Here, the functions above are defined in (\ref{herm}), (\ref{potential D}), (\ref{function g}), (\ref{solution by Yau}) and (\ref{potential F}).\
\end{prop}

\begin{rem}
From a priori estimate due to Ko\l odziej \cite{Ko}, the solution $\varphi$ is bounded on $X$.\
Thus, by taking $c>0$ sufficiently large, $\varphi$ can be ignored when we consider the value of $M_{c,v,\eta}$.\
\end{rem}
By taking a sufficiently large $c>0$, we have
\begin{eqnarray}
\label{global}
M_{c,v,\eta}=\left\{ \begin{array}{ll}
\Theta(t) & \hspace{10pt} \mbox{near $D$ and away from $F$,} \\
\tilde{G}_{v}^{\beta}(b) & \hspace{10pt} \mbox{near $F$ and away from $D$,} \\
t + \varphi +c & \hspace{10pt} \mbox{away from $F$ and  $D$.} \\
\end{array} \right.
\end{eqnarray}

Set
$$
\omega_{c,v,\eta} := \dol M \left( \Theta(t),\tilde{G}_{v}^{\beta}(b), t + \varphi +c \right).
$$

The reason why we consider the second {\kah} potential which contains the term $\kappa \Theta(t)$ is that we want to make $\omega_{c,v,\eta}$ complete on $X \setminus D$.\
The function $M_{c,v,\eta}$ is defined on $X \setminus (D \cup F)$.\
On the other hand, Lemma \ref{gamma} implies that $\omega_{c,v,\eta}$ is defined on $X \setminus D$ since the {\kah} metric $\gamma_{v}^{\beta}$ is a smooth {\kah} metric on $X$.\
From (\ref{global}), we know that the scalar curvature of $\omega_{c,v,\eta}$ is small on three regions above (in particular, away from $D$ and $F$, $S(\omega_{c,v,\eta}) = 0$ since $t + \varphi +c$ is a {\kah} potential whose Ricci form is zero).\

The explicit formula of $\omega_{c,v,\eta}$ is written as
\begin{eqnarray*}
\omega_{c,v,\eta}
&=& \frac{\partial M_{c,v,\eta}}{\partial t_{1}} \omega_{0} + \frac{\partial M_{c,v,\eta}}{\partial t_{2}} ( \gamma_{v}^{\beta} + \kappa \omega_{0} ) + \frac{\partial M_{c,v,\eta}}{\partial t_{3}} \dol (t + \varphi)\\
&\hspace{15pt}+& \left[ 
\begin{array}{ccc}
\partial \Theta(t) & \partial \tilde{G}_{v}^{\beta}(b) &  \partial (t + \varphi)
\end{array} 
\right]
\left[ 
\begin{array}{c}
\frac{\partial^{2} M_{c,v,\eta}}{\partial t_{i} \partial t_{j}}
\end{array} 
\right]
\left[ 
\begin{array}{ccc}
\overline{\partial} \Theta(t) & \overline{\partial} \tilde{G}_{v}^{\beta}(b) & \overline{\partial} (t + \varphi)
\end{array} 
\right]^{t}.
\end{eqnarray*}
Thus, when we compute the scalar curvature of $\omega_{c,v,\eta}$, higher order derivatives of $\varphi$ arise in the components of the Ricci tensor of $\omega_{c,v,\eta}$.\
So, we must study the behavior of higher order derivatives of $\varphi$ near $D \cup F$.\

\section{
Proof of Theorem \ref{solution}}
\label{sec:8}

In this section, we prove Theorem \ref{solution}.\
Firstly, we use the $C^{2}$-estimate due to P$\check{{\rm a}}$un \cite{Pa} (see also \cite{Di}, \cite[p.366, Theorem 14.3]{GZ}) for the solution $\phi$ of the complex {\ma} equation (\ref{solution by Yau}) in the previous section to obtain the estimate of the ellipticity. i.e., the maximal ratio of the maximal eigenvalue to the minimal eigenvalue, of the {\kah} metric $\theta_{X} + \dol \varphi$.\
Secondly, we study how the $C^{2,\epsilon}$-estimate of $\varphi$ depends on the ellipticity of $\theta_{X} + \dol \varphi$ on a fixed relatively compact domain in $X \setminus (D \cup F)$.\
Finally, we estimate the higher order derivatives of $\varphi$ by using the Schauder estimate.\

\subsection{The $C^{2}$-estimate}
To study the behavior of the higher order derivatives of $\varphi$, the elliptic operator defined by the {\kah} metric $\theta_{X} + \dol \varphi$ plays an important role.\
To obtain the ellipticity of $\theta_{X} + \dol \varphi$, we use the $C^{2}$-estimate due to P$\check{{\rm a}}$un \cite{Pa} (see also \cite{Di}, \cite[p.366, Theorem 14.3]{GZ}).\
\begin{thm}
\label{p c2 estimate}
Let $dV$ be a smooth volume form.\
Assume that $\varphi \in {\rm PSH} (X, \theta_{X})$ satisfies
$$
(\theta_{X} + \dol \varphi)^{n} = e^{\psi_{+} - \psi_{-}}dV
$$
with $\displaystyle\int_{X} \varphi \theta_{X}^{n} = 0$.\
Here, $\psi_{+}, \psi_{-}$ are quasi-plurisubharmonic functions on $X$.\
Assume that we are given $C>0$ and $p>1$ such that
\begin{description}
\item[(i)] $\dol \psi_{+} \geq - C \theta_{X}$ and $\sup_{X} \psi_{+} \leq  C$.\
\item[(ii)] $\dol \psi_{-} \geq - C \theta_{X}$ and $|| e^{- \psi_{-}} ||_{L^{p}} \leq  C$.\
\end{description}
Then there exists $A>0$ depending only on $\theta_{X}$, $p$ and $C$ such that
$$
0 \leq \theta_{X} + \dol \varphi \leq A e^{- \psi_{-}} \theta_{X}.
$$
\end{thm}

Set $\psi_{+} := \log || \sigma_{D} ||^{2m/l}$ and $\psi_{-} := \log || \sigma_{F} ||^{2/l}$.\
Then, Theorem \ref{p c2 estimate} implies the following inequality
\begin{equation}
\label{our c2 estimate}
0 \leq \theta_{X} + \dol \varphi \leq A ||\sigma_{F}||^{-2/l} \theta_{X}.
\end{equation}
Recall that the singular and degenerate volume form
\begin{equation}
\label{our ma}
(\theta_{X} + \dol \varphi)^{n} = \xi^{-1/l} \wedge \overline{\xi^{-1/l}}, \hspace{8pt} \xi = \sigma_{F} \otimes \sigma_{D}^{-m}
\end{equation}
vanishes along $D$ with order $2m/l$ and has a pole along $F$ of order $2/l$.\
So, we obtain the behavior of the product of the eigenvalues of the {\kah} metric $\theta_{X} + \dol \varphi$.\
From (\ref{our c2 estimate}) and (\ref{our ma}), we can estimate the eigenvalues of $\theta_{X} + \dol \varphi$.\
Namely, the maximal eigenvalue $\Lambda$ and the minimal eigenvalue $\lambda$ of the {\kah} metric $\theta_{X} + \dol \varphi$ are estimated as follows :
$$
\Lambda = O(|| \sigma_{F} ||^{-2/l} ), \hspace{7pt} \lambda^{-1} = O(|| \sigma_{D} ||^{-2m/l} ).
$$
In the next subsection, to consider the third and the forth order derivatives, we recall the $C^{2,\epsilon}$-estimate of $\varphi$.\

\subsection{
The $C^{2,\epsilon}$-estimate}
This subsection follows from \cite[Chapter 14]{GZ}.\
In this subsection, we study the relation between the ellipticity of $\theta_{X} + \dol \varphi$ and the $C^{2,\epsilon}$-estimate of $\varphi$.\
This subsection is the core of the proof of Theorem \ref{solution} because the estimate of the higher order derivatives of the solution $\varphi$ are obtained by the $C^{2,\epsilon}$-estimate and the Schauder estimate.\

Let $\mathcal{H}$ be the set of all $n \times n$ Hermitian matrices and set
$$
\mathcal{H}_{+} := \{ A \in \mathcal{H} | A > 0 \}.
$$
In addition, for $0 < \lambda < \Lambda < \infty$, let $S(\lambda,\Lambda)$ be the subset of $\mathcal{H}_{+}$ whose eigenvalues lie in the interval $[\lambda,\Lambda]$.\
First, recall the following result from linear algebra (see \cite[p.454, Lemma 17.13]{GT}, \cite[p.372, Lemma 14.10]{GZ}):

\begin{lemma}
\label{orthogonal lemma}
One can find unit vectors $\zeta_{1},...,\zeta_{N} \in \mathbb{C}^{n}$ and $0 < \lambda_{*} < \Lambda_{*} < \infty$, depending only on $n,\lambda$ and $\Lambda$, such that every $A \in S(\lambda,\Lambda)$ can be written as
$$
A = \sum_{k =1}^{N} \beta_{k} \zeta_{k} \otimes \overline{\zeta_{k}}, \hspace{7pt} {\rm i.e.,} \hspace{7pt} a_{i,\overline{j}} = \sum_{k} \beta_{k} \zeta_{ki} \overline{\zeta}_{kj},
$$
where $\beta_{k} \in [\lambda_{*},\Lambda_{*}]$.\
The vectors $\zeta_{1},...,\zeta_{N} \in \mathbb{C}^{n}$ can be chosen so that they contain a given orthonormal basis of $\mathbb{C}^{n}$.\
\end{lemma}

\begin{rem}
\label{uniform N}
In the proof of Lemma \ref{orthogonal lemma}, they use the following covering
$$
U(\zeta_{1},...,\zeta_{n^{2}}) = \left\{ \sum_{k} \beta_{k} \zeta_{k} \otimes \overline{\zeta}_{k} \hspace{5pt} | \hspace{5pt} 0 < \beta_{k} < 2\Lambda \right\}
$$
of the compact subset $S(\lambda/2,\Lambda)$ (see \cite[p.454, Lemma 17.13]{GT}, \cite[p.372, Lemma 14.10]{GZ}).\
Here, $\zeta_{1},...,\zeta_{n^{2}} \in \mathbb{C}^{n}$ are unit vectors such that the matrices $\zeta_{k} \otimes \overline{\zeta}_{k}$ span $\mathcal{H}$ over $\mathbb{R}$.\
Thus, it follows from the form of the covering $U(\zeta_{1},...,\zeta_{n^{2}})$ that the number $N$ in Lemma \ref{orthogonal lemma} is depending only on the dimension $n$.\
In particular, $N$ is independent of the ellipticity of $\theta_{X} + \dol \varphi$.\
\end{rem}

Take local holomoriphic coordinates $(z^{i})_{i=1}^{n} = (z^{1},z^{2},...,z^{n-2},w_{F},w_{D})$ such that $\{ w_{F} = 0 \} = F$ and $\{ w_{D} = 0 \} = D$.\
On this coordinate chart, we can write $t = a + \log |w_{D}|^{-2}$ for some smooth plurisubharmonic function $a$.\
Since $\theta_{X} + \dol \varphi = \dol (a + \varphi)$ on this coordinate chart, it is enough to consider the following complex Monge-Amp$\grave{{\rm e}}$re equation 
$$
\det (u_{i,\overline{j}}) = f
$$
on an open subset $\Omega \Subset \mathbb{C}^{n} \setminus (D \cup F) $ by setting
\begin{equation}
\label{u and phi}
u = a + \varphi.
\end{equation}
It follows from our construction that we may assume that the function $f$ is a form of
$$
f = |w_{F}|^{-2/l} |w_{D}|^{2m/l}.
$$
Fix an unit vector $\zeta \in \mathbb{C}^{n}$.\
Differentiating the following equation :
$$
\log \det (u_{i,\overline{j}}) = \log f,
$$
we have
$$
u^{i,\overline{j}} u_{\zeta,\overline{\zeta},i,\overline{j}} = (\log f)_{\zeta,\overline{\zeta}} + u^{i,\overline{l}}u^{k,\overline{j}} u_{\zeta,i,\overline{j}} u_{\overline{\zeta},k,\overline{l}} \geq (\log f)_{\zeta,\overline{\zeta}}=0.
$$
Here we use the standard Einstein convention and the notation $(u^{i,\overline{j}}) = ((u_{i,\overline{j}})^{t})^{-1}$.\
Set 
$$
a^{i,\overline{j}} = f u^{i,\overline{j}}.
$$
Then, for any $i$, we have
$$
(a^{i,\overline{j}})_{\overline{j}} = f_{\overline{j}} u^{i,\overline{j}} - f u^{i,\overline{l}}u^{k,\overline{j}} u_{\overline{j},k,\overline{l}} = f u^{k,\overline{l}} u_{\overline{j},k,\overline{l}} u^{i,\overline{j}} - f u^{i,\overline{l}}u^{k,\overline{j}} u_{\overline{j},k,\overline{l}} = 0.
$$
Thus, we obtain
$$
(a^{i,\overline{j}} u_{\zeta,\overline{\zeta},i})_{\overline{j}} = (a^{i,\overline{j}})_{\overline{j}} u_{\zeta,\overline{\zeta},i} + a^{i,\overline{j}} u_{\zeta,\overline{\zeta},i,\overline{j}} \geq f  (\log f)_{\zeta,\overline{\zeta}} = 0.
$$
Note that $u_{\zeta,\overline{\zeta}}$ is a subsolution of the equation $L v = 0$, where $L v := \sum_{i,j} (a^{i,\overline{j}} v_{i})_{\overline{j}}$.\
The assumption of $u$ and the later lemma ensure that the operator $L$ is {\it uniformly} elliptic (in the real sense).\
Then, we have the following estimate (see \cite[Theorem 8.18]{GT}).

\begin{lemma}
\label{Harnack constant}
The weak Harnack inequality
$$
r^{-2n} \int_{B_{r}} (\sup_{B_{4r}} u_{\zeta,\overline{\zeta}} - u_{\zeta,\overline{\zeta}}) \leq C_{H} (\sup_{B_{4r}} u_{\zeta,\overline{\zeta}} - \sup_{B_{r}} u_{\zeta,\overline{\zeta}}),
$$
holds.\
Here, $B_{4r} := B(z_{0},4r) \subset \Omega$ with $d(z_{0}, \partial \Omega ) > 4r$.\
Moreover, in our case, we have the following estimate of the constant $C_{H}$ in Harnack inequality :
$$
C_{H} = O( \Lambda/\lambda ).
$$
\end{lemma}

\begin{proof}
It suffices to show the estimate of the constant $C_{H}$.\
In our case, we will only consider the behavior of $\varphi$ in the neighborhood of $D \cup F$ and the $C^{2}$-estimate of $\varphi$ implies that
\begin{eqnarray*}
u_{\zeta,\overline{\zeta}} &=& O(|| \sigma_{F} ||^{-2/l}) = O(\Lambda)\\
u_{\zeta,\overline{\zeta}}^{-1} &=& O(|| \sigma_{D} ||^{-2m/l}) = O(\lambda^{-1})
\end{eqnarray*}
as $|| \sigma_{F} || \to 0$ and $|| \sigma_{D} || \to 0$.\
Thus, the weak Harnack inequality implies that the lemma follows.\
\end{proof}

\begin{rem}
From the proof of \cite[Theorem 8.18]{GT}, we know that the optimal Harnack constant $C_{H}$ is estimated by
$$
C_{H} = C_{n}^{\sqrt{\Lambda/\lambda}},
$$
where $C_{n}$ depends only on $n$.\
\end{rem}

Set $U := (u_{i,\overline{j}})$.\
For $x,y \in B_{4r}$, we obtain
$$
a^{i,\overline{j}}(y)u_{i,\overline{j}}(x) = f(y) u^{i,\overline{j}}(y) u_{i,\overline{j}}(x) = f(y) {\rm tr}(U(y)^{-1}U(x)).
$$
In particular, $a^{i,\overline{j}}(y)u_{i,\overline{j}}(y) = n f(y)$.\
Since $\det (f(y)^{1/n}U(y)^{-1}) = 1$, we have
\begin{eqnarray*}
a^{i,\overline{j}}(y)u_{i,\overline{j}}(x)
&=& f(y)^{1 - 1/n} {\rm tr}(f(y)^{1/n} U(y)^{-1}U(x))\\
&\geq& n f(y)^{1 - 1/n}\det(U(x))^{1/n}\\
&=& n f(y)^{1 - 1/n}f(x)^{1/n}.
\end{eqnarray*}
Here, we have used the following lemma (see \cite[Lemma 5.8]{GZ}) :
\begin{lemma}
For any $A \in \mathcal{H}_{+}$, we have
$$
(\det A)^{1/n} = \frac{1}{n} \inf \{ {\rm tr}(AB) | B \in \mathcal{H}_{+}, \det B =1 \}.
$$
\end{lemma}
Therefore, for any $x,y \in B_{4r}$ and $\epsilon \in (0,1)$, we have
\begin{eqnarray*}
a^{i,\overline{j}}(y)(u_{i,\overline{j}}(y) - u_{i,\overline{j}}(x))
&\leq& n f(y) - n f(y)^{1 - 1/n}f(x)^{1/n}\\
&=& n f(y)^{1 - 1/n} (f(y)^{1/n} - f(x)^{1/n})\\
&\leq& C(\epsilon)_{4} |x-y|^{\epsilon},
\end{eqnarray*}
where
$$
C(\epsilon)_{4} := n \sup_{\Omega} (f^{1 - 1/n}) \mbox{ H$\ddot{{\rm o}}$l}_{\epsilon,\Omega}(f^{1/n})
$$
and $ \mbox{ H$\ddot{{\rm o}}$l}_{\epsilon,\Omega}$ denotes an $\epsilon$-{\hol} constant.\
In this case, the following estimates
\begin{eqnarray}
\mbox{ H$\ddot{{\rm o}}$l}_{\epsilon,\Omega}(f^{1/n})
&=& O(||\sigma_{F}||^{ -2/nl - \epsilon} ||\sigma_{D}||^{2m/nl - \epsilon})\\
\sup_{\Omega} (f^{1 - 1/n})
&=& O(||\sigma_{F}||^{ -2(n-1)/nl} ||\sigma_{D}||^{2m(n-1)/nl})
\end{eqnarray}
implies that we have
\begin{eqnarray}
\label{e holder}
C(\epsilon)_{4} = O(||\sigma_{F}||^{ -2/l - \epsilon} ||\sigma_{D}||^{2m/l - \epsilon}).
\end{eqnarray}

\begin{rem}
In \cite[p.375]{GZ}, they used the Lipscitz constant of $f$.\
But in our case, it is enough to use the {\hol} constant of $f$ for sufficiently small $\epsilon$.\
\end{rem}

Set $\lambda, \Lambda >0$ so that the eigenvalues of $(a^{i,\overline{j}}(y))$ lie in the interval $[\lambda, \Lambda]$.\
Then, Lemma \ref{orthogonal lemma} implies that we can find unit vectors $\zeta_{1},...,\zeta_{N} \in \mathbb{C}^{n}$ such that for any $x,y \in \Omega$,
$$
a^{i,\overline{j}}(y)(u_{i,\overline{j}}(y) - u_{i,\overline{j}}(x)) = \sum_{k=1}^{N} \beta_{k}(y) (u_{\zeta_{k},\overline{\zeta}_{k}}(y) - u_{\zeta_{k},\overline{\zeta}_{k}}(x)),
$$
where $\beta_{k}(y) \in [\lambda_{*},\Lambda_{*}]$ and $\lambda_{*},\Lambda_{*}>0$.\

Thus, we have
$$
\sum_{k=1}^{N} \beta_{k}(y) (u_{\zeta_{k},\overline{\zeta}_{k}}(y) - u_{\zeta_{k},\overline{\zeta}_{k}}(x)) \leq C(\epsilon)_{4} |x-y|^{\epsilon}.
$$

Set
$$
M_{k,r} := \sup_{B_{r}} u_{\zeta_{k},\overline{\zeta}_{k}} , \hspace{8pt} m_{k,r} := \inf_{B_{r}} u_{\zeta_{k},\overline{\zeta}_{k}},
$$
and
$$
\eta(r) := \sum_{k=1}^{N}(M_{k,r}-m_{k,r}).
$$
To establish the {\hol} condition
$$
\eta(r) \leq C r^{\tilde{\epsilon}}
$$
for some $0<\tilde{\epsilon}<1$, we need the following lemma from \cite[p.201, Lemma 8.23]{GT} :
\begin{lemma}
\label{small lemma}
Let $\eta$ and $\sigma$ be non-decreasing functions defined on the interval $(0,R_{0}]$ such that there exist $\tau, \alpha \in (0,1)$ satisfying
$$
\eta(\tau r) \leq \alpha \eta(r) + \sigma(r)
$$
for all $r \in (0,R_{0}]$.\
Then, for any $\mu \in (0,1)$, we have

$$
\eta(R) < \frac{1}{\alpha} \left( \frac{R}{R_{0}} \right)^{(1-\mu)(\log \alpha / \log \tau)} + \frac{\sigma(R_{0}^{1-\mu} R^{\mu})}{1 - \alpha}.
$$
\end{lemma}
So, it suffices to show that
$$
\eta(r) \leq \delta \eta(4r) + C r^{\epsilon} , \hspace{6pt} 0 < r < r_{0},
$$
where $\delta, \epsilon \in (0,1)$ and $r_{0}>0$.\

For fixed $k$, Harnack inequality implies that
\begin{eqnarray*}
r^{-2n} \int_{B_{r}} \sum_{l \neq k} ( M_{l,4r} - u_{\zeta_{l},\overline{\zeta}_{l}} )
&=& \sum_{l \neq k} r^{-2n} \int_{B_{r}} ( M_{l,4r} - u_{\zeta_{l},\overline{\zeta}_{l}} )\\
&\leq& \sum_{l \neq k} C_{H} ( M_{l,4r} - M_{l,r})\\
&\leq& \sum_{l \neq k} C_{H} ( \eta(4r) - \eta(r) )\\
&=& (N-1) C_{H} ( \eta(4r) - \eta(r)).
\end{eqnarray*}

For $x \in B_{4r}$ and $y \in B_{r}$, we have
\begin{eqnarray*}
\beta_{k}(y)  (u_{\zeta_{k},\overline{\zeta}_{k}}(y) - u_{\zeta_{k},\overline{\zeta}_{k}}(x))
&\leq& C(\epsilon)_{4} |x-y|^{\epsilon} + \sum_{l \neq k} \beta_{l}(y)  (u_{\zeta_{l},\overline{\zeta}_{l}}(x) - u_{\zeta_{l},\overline{\zeta}_{l}}(y))\\
&\leq& 5 C(\epsilon)_{4} r^{\epsilon} + \Lambda_{*} \sum_{l \neq k}  (M_{l,4r} - u_{\zeta_{l},\overline{\zeta}_{l}}(y)).\\
\end{eqnarray*}

Thus, for all $y \in B_{r}$, we have
$$
u_{\zeta_{k},\overline{\zeta}_{k}}(y) - m_{k,4r} \leq \frac{1}{\lambda_{*}} \left(5 C(\epsilon)_{4} r^{\epsilon} + \Lambda_{*} \sum_{l \neq k}  (M_{l,4r} - u_{\zeta_{l},\overline{\zeta}_{l}}(y)) \right).
$$

Therefore,
\begin{eqnarray*}
r^{-2n} \int_{B_{r}} ( u_{\zeta_{k},\overline{\zeta}_{k}}(y) - m_{k,4r} )
&\leq& r^{-2n} \int_{B_{r}}  \frac{1}{\lambda_{*}} \left(5 C(\epsilon)_{4} r^{\epsilon} + \Lambda_{*} \sum_{l \neq k}  (M_{l,4r} - u_{\zeta_{l},\overline{\zeta}_{l}}(y)) \right)\\
&\leq& \frac{5C(\epsilon)_{4}}{\lambda_{*}}r^{\epsilon} + \frac{\Lambda_{*}}{\lambda_{*}} r^{-2n} \int_{B_{r}} \sum_{l \neq k} ( M_{l,4r} - u_{\zeta_{l},\overline{\zeta}_{l}} )\\
&\leq& \frac{5C(\epsilon)_{4}}{\lambda_{*}}r^{\epsilon} + \frac{\Lambda_{*}}{\lambda_{*}} (N-1) C_{H} ( \eta(4r) - \eta(r) ).
\end{eqnarray*}

Using Harnack inequality again, we have
\begin{eqnarray*}
M_{k,4r} - m_{k,4r}
&=& r^{-2n} \int_{B_{r}} (\sup_{B_{4r}} u_{\zeta_{k},\overline{\zeta}_{k}} - u_{\zeta_{k},\overline{\zeta}_{k}}) + r^{-2n} \int_{B_{r}} ( u_{\zeta_{k},\overline{\zeta}_{k}}(y) - m_{k,4r} )\\
&\leq& C_{H} (M_{k,4r} - M_{k,r} ) + \frac{5C(\epsilon)_{4}}{\lambda_{*}}r^{\epsilon} + \frac{\Lambda_{*}}{\lambda_{*}} (N-1) C_{H} ( \eta(4r) - \eta(r) )\\
&\leq& \left( C_{H} + \frac{\Lambda_{*}}{\lambda_{*}} (N-1) C_{H} \right) \eta(4r) \\
&\hspace{20pt}& - \left( C_{H} + \frac{\Lambda_{*}}{\lambda_{*}} (N-1) C_{H} \right) \eta(r) + \frac{5C(\epsilon)_{4}}{\lambda_{*}} r^{\epsilon}.
\end{eqnarray*}

Summing over $k$, we have
\begin{eqnarray*}
\eta(4r)
&\leq& N \left( C_{H} + \frac{\Lambda_{*}}{\lambda_{*}} (N-1) C_{H} \right) \eta(4r)\\
&\hspace{20pt}& - N \left( C_{H} + \frac{\Lambda_{*}}{\lambda_{*}} (N-1) C_{H} \right) \eta(r) + N \frac{5C(\epsilon)_{4}}{\lambda_{*}} r^{\epsilon}.
\end{eqnarray*}
Thus, we obtain
\begin{eqnarray}
\label{consequence}
\eta(r)
&\leq& \frac{N \left( C_{H} + \frac{\Lambda_{*}}{\lambda_{*}} (N-1) C_{H} \right) - 1}{N \left( C_{H} + \frac{\Lambda_{*}}{\lambda_{*}} (N-1) C_{H} \right)} \eta(4r) + \frac{ \frac{5C(\epsilon)_{4}}{\lambda_{*}} }{ C_{H} + \frac{\Lambda_{*}}{\lambda_{*}} (N-1) C_{H} }r^{\epsilon}.
\end{eqnarray}
Since we can take arbitrary $\lambda^{*}N < \lambda$ and $\Lambda^{*} > \Lambda$, we may assume that $\lambda^{*}N = \lambda$ and $\Lambda^{*} = \Lambda$.\
Thus, we have
\begin{lemma}
By taking $\epsilon \leq 2/l$, there exists $0 < \tilde{\epsilon} <\epsilon$ with
$$
|| u ||_{C^{2,\tilde{\epsilon}}} =  O\left( \left( \frac{\Lambda}{\lambda} \right) C_{H} \right).
$$
\end{lemma}

\begin{proof}
In order to show this lemma, we apply Lemma \ref{small lemma} to the inequality (\ref{consequence}).\
Set
$$
\alpha :=  \frac{N \left( C_{H} + \frac{\Lambda}{\lambda} (N-1) C_{H} \right) - 1}{N \left( C_{H} + \frac{\Lambda}{\lambda} (N-1) C_{H} \right)},
$$
where this is the coefficient of $\eta(4r)$ in (\ref{consequence}).\
Then, we have the following estimates :
$$
\frac{1}{\alpha} = O(1), \hspace{5pt} \frac{1}{1-\alpha} = O((\Lambda/\lambda)C_{H}).
$$
Here, we have used the fact that the number $N$ depends only on the dimension $n$ (Remark \ref{uniform N}).\
Define a non-decreasing function $\sigma$ by
$$
\sigma(r) := \frac{ \frac{5C(\epsilon)_{4}}{\lambda} }{ C_{H} + \frac{\Lambda}{\lambda} (N-1) C_{H} }r^{\epsilon}.\
$$
Here, this is the second term in the right hand side of the inequality (\ref{consequence}).\
Recall the estimate (\ref{e holder})
$$
C(\epsilon)_{4} = O(||\sigma_{F}||^{ -2/l - \epsilon} ||\sigma_{D}||^{2m/l - \epsilon})
$$
and Lemma \ref{Harnack constant}.\
The assumption that $\epsilon \leq 2/l$ implies that we have the following
$$
\frac{ \frac{5C(\epsilon)_{4}}{\lambda} }{ C_{H} + \frac{\Lambda}{\lambda} (N-1) C_{H} } = O(1).
$$
Lemma \ref{small lemma} implies that we have
$$
\eta(r) < \frac{1}{\alpha} \left( \frac{r}{r_{0}} \right)^{(1-\mu)(\log \alpha / \log (1/4))} + \frac{\sigma(r_{0}^{1-\mu} r^{\mu})}{1 - \alpha},
$$
for any $\mu \in (0,1)$.\
Take $\mu \in (0,1)$ so that
$$
(1 - \mu )(\log \alpha / \log (1/4)) > \mu \epsilon.
$$
Thus, we have
$$
\eta(r) < O((\Lambda/\lambda)C_{H})\sigma(r_{0}^{1-\mu} r^{\mu})
$$
Set $\tilde{\epsilon} := \epsilon \mu < \epsilon$.\
From the interior {\hol} estimate for solutions of Poisson's equation \cite[Theorem 4.6]{GT}, we finish the proof.
\end{proof}

Recall the relation (\ref{u and phi}) between $u$ and $\varphi$.\
Lemma \ref{Harnack constant} implies
\begin{prop}
\label{c2 estimate}
For the domain $\Omega \Subset X \setminus (D \cup F)$, we have
$$
|| \varphi ||_{C^{2,\tilde{\epsilon}}(\Omega)} = O\left( \left( || \sigma_{D} ||^{-2m/l} || \sigma_{F} ||^{-2/l} \right)^{2} \right)
$$
as $\sigma_{D}, \sigma_{F} \to 0$.\
\end{prop}

\subsection{
The third and the forth order estimates}

In this subsection, we prove Theorem \ref{solution}.\
This subsection also follows from \cite[Chapter 14]{GZ}.\
To consider higher order estimates, we recall the Schauder estimate with respect to the elliptic linear operator defined by the {\kah} metric $\theta_{X} + \dol \varphi$.\
The complex Monge-Amp$\grave{{\rm e}}$re operator
$$
F(D^{2} u) = \det (u_{i,\overline{j}})
$$
is elliptic if the $2n \times 2n$ real symmetric matrix $A := (\partial F / \partial u_{p,q})$ is positive (we denote here by $u_{p,q}$ the element of the real Hessian $D^{2} u$).\
The matrix $A$ is determined by
$$
\frac{d}{dt} F (D^{2} u + t B)|_{t=0} = {\rm tr }(A^{t}B).
$$
From \cite{Bl} (see also \cite[Exercise 14.8]{GZ}), we have
\begin{lemma}
One has
$$
\lambda_{{\rm min}}( \partial F / \partial u_{p,q}) = \frac{ \det (u_{i,\overline{j}})}{ 4\lambda_{{\rm max}}(u_{i,\overline{j}})}, \hspace{6pt} \lambda_{{\rm max}}( \partial F / \partial u_{p,q}) = \frac{ \det (u_{i,\overline{j}})}{ 4\lambda_{{\rm min}}(u_{i,\overline{j}})},
$$
where $\lambda_{{\rm min}}( \partial F / \partial u_{p,q})$ and $\lambda_{{\rm max}}( \partial F / \partial u_{p,q})$ denote minimal and maximal eigenvalue of the matrix $( \partial F / \partial u_{p,q}))_{p.q}$ respectively.\
\end{lemma}
Then, we can estimate the ellipticity in the real sense.\
We apply the standard elliptic theory to the equation
$$
F(D^{2} u) = f.
$$
For a fixed unit vector $\zeta$ and small $h>0$, we consider
$$
u^{h}(x) := \frac{u(x+h\zeta) - u(x)}{h}
$$
and
$$
a_{h}^{p,q}(x) := \int^{1}_{0} \frac{\partial F}{\partial u_{p,q}} (t D^{2} u(x + h \zeta) + (1-t) D^{2}u(x)) dt.
$$
Thus, we have
$$
a_{h}^{p,q}(x) u^{h}_{p,q}(x) = \frac{1}{h} \int_{0}^{1} \frac{d}{dt} F (t D^{2} u(x + h \zeta) + (1-t) D^{2}u(x)) dt = f^{h}(x).
$$
From the definition of $a_{h}^{p,q}$, we obtain
$$
|| a_{h}^{p,q} ||_{C^{0,\tilde{\epsilon}}} \leq C || u ||_{C^{2,\tilde{\epsilon}}}^{n-1} = O((\Lambda/\lambda)^{2(n-1)})
$$
for sufficiently small $h>0$.\

The Schauder estimate implies
\begin{prop}
\label{higher derivative}
There exists $C_{S} > 0$ such that
$$
|| u^{h} ||_{C^{2,\tilde{\epsilon}}} \leq C_{S} ( || f^{h} ||_{C^{0,\tilde{\epsilon}}} + || u^{h} ||_{C^{0}} )
$$
for any $h>0$.\
\end{prop}

Therefore, we can obtain the estimate of derivatives of the solution $\varphi$ in the desired direction by taking a suitable vector $\zeta$ and $h \to 0$.\
The constant $C_{S}$ in Proposition \ref{higher derivative} also depends on the maximal ratio of the eigenvalues $\Lambda / \lambda$ and the dimension $n$.\
By examining the proof of \cite[Lemma 6.1 and Theorem 6.2]{GT}, there is a positive constant $s(n)$ depending only on the dimension $n$ such that
$$
C_{S} = O ((\Lambda/\lambda)^{s(n)}).
$$

As $h \to 0$, we have the following third order estimates of $\varphi$ :
\begin{prop}
For any multi-index $\alpha = (\alpha_{1},...,\alpha_{n})$ satisfying $\sum_{i} \alpha_{i} = 2$, we have\
$$
\left| \frac{\partial}{\partial z^{i} } \partial^{\alpha} \varphi \right| = O \left( C_{S} |w_{D}|^{-4 m/l} |w_{F}|^{-4 /l} \right),
$$
$$
\left| \frac{\partial}{\partial w_{F} } \partial^{\alpha} \varphi \right| = O \left( C_{S}|w_{D}|^{-4 m/l} |w_{F}|^{-1 -4 /l} \right),
$$
$$
\left| \frac{\partial}{\partial w_{D} } \partial^{\alpha} \varphi \right| = O \left( C_{S} |w_{D}|^{-1 -4 m/l} |w_{F}|^{-4 /l} \right),
$$
as $|w_{D}|,|w_{F}| \to 0$.\
\end{prop}

From the discussion above, we can prove Theorem \ref{solution}.\

\medskip
{\it Proof of Theorem \ref{solution}}
Let $\dot{a}_{h}^{p,q}$ be a differential of $a_{h}^{p,q}$ in some direction.\
From the definition of $a_{h}^{p,q}$, we know that
$$
|| \dot{a}_{h}^{p,q} ||_{C^{0,\tilde{\epsilon}}} \leq C || \dot{u} ||_{C^{2,\tilde{\epsilon}}} || u ||_{C^{2,\tilde{\epsilon}}}^{n-2}.
$$

Thus, by differentiating the equation $a_{h}^{p,q}(x) u^{h}_{p,q}(x) = f^{h}(x)$, Schauder estimate implies again the following inequality:
\begin{eqnarray*}
|| \dot{u}^{h} ||_{C^{2,\tilde{\epsilon}}}
&\leq& C_{S} ( || \dot{f}^{h} - \dot{a}_{h}^{p,q} u^{h}_{p,q} ||_{C^{0,\tilde{\epsilon}}} + || \dot{u}^{h} ||_{C^{0}} ).
\end{eqnarray*}
Thus, we finish the proof pf Theorem \ref{solution} by taking a suitable vector $\zeta$ and $h \to 0$.\
\sq

\begin{rem}
By examining the proof of \cite[Lemma 6.1 and Theorem 6.2]{GT} and the discussion above, we can find that
$$
a = a(n) = O(n^{2}).
$$
\end{rem}

\section{
Proof of Theorem \ref{small scalar curvature}}
\label{sec:9}

In this section, we prove Theorem \ref{small scalar curvature}.\
To compute the scalar curvature of the {\kah} metric $\omega_{c,v,\eta}$, we have to consider the inverse matrix (see Lemma 3.4 in \cite{Aoi1}).\
Since we assume that the divisor $D+F$ is simple normal crossing, we can choose block matrices in suitable directions in local holomorphic coordinates defining hypersurfaces $D$ and $F$.\
To prove Theorem \ref{small scalar curvature}, we consider the case that the parameter $\eta = (\eta_{1},\eta_{2},\eta_{3})$ depends on $c>0$.\
More precisely, we set $\eta_{i} := a_{i} c$ for $i=1,2$ for $a_{i} \in (0,1)$ and $\eta_{3}$ a fixed positive real number.\
We use many parameters, i.e., $c,v,\beta, \kappa,\eta,a_{i}$.\
When we want to make the scalar curvature $S(\omega_{c,v,\eta})$ small, we take sufficiently large $c$ and sufficiently small $v$.\
On the other hand, we don't make other parameters $\beta, \kappa,a_{i}$ close to $\infty$, 0 or 1.\
Namely, the parameters $\beta, \kappa,a_{i}$ are bounded in this sense.\
Settings of these bounded parameters will be given later.\

\medskip
{\it Proof of Theorem \ref{small scalar curvature}.}
Take a relatively compact domain $Y \Subset X \setminus (D \cup F)$.\
Recall that the function $G_{v}^{\beta}(\beta b)$ is defined by
$$
G_{v}^{\beta}(\beta b) := \int_{b_{0}}^{\beta b} \left( \frac{1}{e^{-y} + v} \right)^{1/\beta} dy.
$$
Immediately, we have $G_{v}^{\beta}(\beta b) < \beta e^{b}$ and $G_{v}^{\beta}(\beta b) \to \beta e^{b}$ as $v \to 0$.\
So, we can find a sufficiently large number $c_{0} = c_{0}(Y) > 0$ so that
$$
Y \Subset \left\{ t + \varphi + c_{0} > \max \{ \Theta(t), \tilde{G}_{v}^{\beta}(b) \} \right\} \Subset X \setminus (D \cup F)
$$
for any $v>0$.\
Here, $\tilde{G}_{v}^{\beta}(b) = G_{v}^{\beta}(\beta b) + \kappa \Theta(t)$.\
For simplicity, we write $\varphi + c_{0}$ by the same symbol $\varphi$.\

Recall that the property d) of the regularized maximum in Lemma \ref{gluing lemma}.\
If the following inequality
$$
\max_{j \neq k} \{ t_{j} + \eta_{j} \} < t_{k} - \eta_{k}
$$
holds for some $k$, we have $M_{\eta}(t) = t_k$.\
For instance, in our case, if we consider the region defined by the following inequality
$$
\max \{ \tilde{G}_{v}^{\beta}(\beta b) + \eta_{2}, t + \varphi + c + \eta_{3} \} < \Theta(t) - \eta_{1},
$$
we have $M_{c,v,\eta} = \Theta(t)$.\
Note that this region is contained in a sufficiently small neighborhood of $D$.\
In this case, we don't have to estimate the scalar curvature $S(\omega_{c,v,\eta})$ since $S(\omega_{c,v,\eta}) = S(\omega_0)$ on this region and the estimate of $S(\omega_0)$ have been obtained in Lemma \ref{estimate zero} before.\
Similarly, if the value of $M_{c,v,\eta}$ corresponds to one of the other variables $\tilde{G}_{v}^{\beta}(b), t+\varphi+c$, Lemma \ref{estimate g} and the Ricci-flatness of the {\kah} metric $\dol (t + \varphi)$ implies that $S(\omega_{c,v,\eta})$ is under control on such regions.\
Thus, it suffices for us to study the $S(\omega_{c,v,\eta})$ on the other regions defined by the inequalities
\begin{eqnarray*}
t_k + \eta_{k} &<& \max_{j \neq k} \{ t_{j} - \eta_{j} \},\\
| t_i - t_j | &<& \eta_{i} + \eta_{j},
\end{eqnarray*}
for $i,j \neq k$ and
\begin{eqnarray*}
| t_1 -  t_2 | &<& \eta_{1} + \eta_{2},\\
| t_2 -  t_3 | &<& \eta_{2} + \eta_{3},\\
| t_1 -  t_3 | &<& \eta_{1} + \eta_{3}.
\end{eqnarray*}
So we have to study $S(\omega_{c,v,\eta})$ on four regions defined by the inequalities above.\

Directly, we have
\begin{eqnarray*}
\omega_{c,v,\eta}
&=& \sqrt{-1} g_{i,\overline{j}} d z^{i} \wedge d \overline{z}^{j}\\
&=& \frac{\partial M_{c,v,\eta}}{\partial t_{1}} \omega_{0} + \frac{\partial M_{c,v,\eta}}{\partial t_{2}} ( \gamma_{v}^{\beta} + \kappa \omega_{0} ) + \frac{\partial M_{c,v,\eta}}{\partial t_{3}} \dol (t + \varphi)\\
&\hspace{15pt}+& \left[ 
\begin{array}{ccc}
\partial \Theta(t) & \partial \tilde{G}_{v}^{\beta}(b) &  \partial (t + \varphi)
\end{array} 
\right]
\left[ 
\begin{array}{c}
\frac{\partial^{2} M_{c,v,\eta}}{\partial t_{i} \partial t_{j}}
\end{array} 
\right]
\left[ 
\begin{array}{ccc}
\overline{\partial} \Theta(t) & \overline{\partial} \tilde{G}_{v}^{\beta}(b) & \overline{\partial} (t + \varphi)
\end{array} 
\right]^{t}.
\end{eqnarray*}
It follows from the convexity of $M_{\eta}$ that the last term is semi-positive.\
When we compute the scalar curvature of $\omega_{c,v,\eta}$, the difficulty comes from terms $\partial \Theta(t) \wedge \overline{\partial} \Theta(t)$ and $\partial \tilde{G}_{v}^{\beta}(b) \wedge \overline{\partial} \tilde{G}_{v}^{\beta}(b) $.\
For these terms, since functions $t$ and $b$ are defined by Hermitian norms of holomorphic sections, it suffices to focus on derivatives in normal directions of smooth hypersurfaces $D$ and $F$ by taking suitable local trivializations of line bundles $L_{X}$ and $K_{X}^{-l} \otimes L_{X}^{m}$ respectively.\
The reason why scalar curvatures of two {\kah} metrics $\omega_{0}, \gamma_{v}^{\beta}$ are under control near these hypersurfaces $D,F$ is that Ricci curvatures are bounded and {\kah} metrics grow asymptotically near these hypersurfaces.\
Thus, it suffices for us to focus on derivatives of $\varphi$ and $M_{\eta}$ arising in Ricci tensors.\
The higher order derivatives of $\varphi$ are estimated in the previous section (Theorem \ref{solution}).\
In addition, the definition of a parameter $\eta = (\eta_{i}) =(a_{1}c, a_{2}c, \eta_{3})$ and Lemma \ref{M derivatives} imply that the higher order derivatives in the first or the second variable of $M_{\eta}$ are estimated by some negative power of $c>0$.\
To estimate $S(\omega_{c,v,\eta})$ on each region, we divide the proof of Theorem \ref{small scalar curvature} into the following four claims.\

\begin{claim}
On the region defined by
\begin{eqnarray*}
( t + \varphi +c) + \eta_{3} &<& \max \{\Theta(t) - \eta_{1}, \tilde{G}_{v}^{\beta}(b) - \eta_{2} \},\\
| \Theta(t) -  \tilde{G}_{v}^{\beta}(b)  | &<& \eta_{1} + \eta_{2},
\end{eqnarray*}
we can make the scalar curvature $S(\omega_{c,v})$ small arbitrarily by taking a sufficiently large $c$.\
\end{claim}

\begin{proof}
On this region, we can write as
\begin{eqnarray*}
\omega_{c,v,\eta}
&=& \frac{\partial M_{c,v,\eta}}{\partial t_{1}} \omega_{0} + \frac{\partial M_{c,v,\eta}}{\partial t_{2}} ( \gamma_{v}^{\beta} + \kappa \omega_{0} )\\
&\hspace{15pt}+& \left[ 
\begin{array}{cc}
\partial \Theta(t) & \partial \tilde{G}_{v}^{\beta}(b)
\end{array} 
\right]
\left[ 
\begin{array}{c}
\frac{\partial^{2} M_{c,v,\eta}}{\partial t_{i} \partial t_{j}}
\end{array} 
\right]
\left[ 
\begin{array}{cc}
\overline{\partial} \Theta(t) & \overline{\partial} \tilde{G}_{v}^{\beta}(b)
\end{array} 
\right]^{t}.
\end{eqnarray*}
To prove this claim, we need the following lemma.\
\begin{lemma}
\label{quadratic terms}
Take a point $p \in D \cap F$ and local holomorphic coordinates $(z^{1},...,z^{n-2},w_{F},w_{D})$ centered at $p$ satisfying $D = \{ w_{D} = 0 \}$ and $F = \{ w_{F} = 0 \}$.\
By taking suitable local trivializations of $L_{X}$ and $K_{X}^{-l} \otimes L_{X}^{m}$, we may assume that if $(z^{1},...,z^{n-2},w_{F},w_{D}) = (0,...,0,w_{F},w_{D})$, we have
\begin{eqnarray*}
\partial \Theta(t) \wedge \overline{\partial} \Theta(t)
&=& O(|w_{F}|^{2}|w_{D}|^{- 4\hat{S}_{D}/n(n-1)}) dw_{F} \wedge d \overline{w_{F}}\\
&\hspace{15pt}+& O(|w_{F}||w_{D}|^{ -1- 4\hat{S}_{D}/n(n-1)}) ( dw_{D} \wedge d \overline{w_{F}} + dw_{F} \wedge d \overline{w_{D}})\\
&\hspace{15pt}+& O(|w_{D}|^{- 2 - 4\hat{S}_{D}/n(n-1)}) dw_{D} \wedge d \overline{w_{D}},\\
\partial G_{v}^{\beta}(\beta b) \wedge \overline{\partial} G_{v}^{\beta}(\beta b)
&=& O((|w_{F}|^{2 \beta} + v)^{-2/\beta} |w_{F}|^{-2}) dw_{F} \wedge d \overline{w_{F}}\\
&\hspace{15pt}+& O(|w_{F}|^{-1}|w_{D}|(|w_{F}|^{2 \beta} + v)^{-2/\beta}) ( dw_{D} \wedge d \overline{w_{F}} + dw_{F} \wedge d \overline{w_{D}})\\
&\hspace{15pt}+& O((|w_{F}|^{2 \beta} + v)^{-2/\beta} |w_{D}|^{2}) dw_{D} \wedge d \overline{w_{D}}.\\
\end{eqnarray*}
\end{lemma}
From the definition of this region, we obtain
\begin{eqnarray*}
\omega_{c,v,\eta}
=
\left[ 
\begin{array}{ccccc}
g_{1,\overline{1}} & \cdots & g_{1,\overline{n-2}} & g_{1,\overline{n-1}} & g_{1,\overline{n}} \\
\vdots & \ddots & \vdots& \vdots & \vdots \\
g_{n-2,\overline{1}}& \cdots &g_{n-2,\overline{n-2}}&g_{n-2,\overline{n-1}}&g_{n-2,\overline{n}}\\
g_{n-1,\overline{1}} & \cdots & g_{n-1,\overline{n-2}} & (|w_{F}|^{2 \beta} + v)^{-2/\beta} |w_{F}|^{-2} & |w_{F}|^{-1}|w_{D}|(|w_{F}|^{2 \beta} + v)^{-2/\beta}\\
g_{n,\overline{1}} & \cdots & g_{n,\overline{n-2}} & |w_{F}|^{-1}|w_{D}|(|w_{F}|^{2 \beta} + v)^{-2/\beta} & |w_{D}|^{- 2 - 4\hat{S}_{D}/n(n-1)}
\end{array} 
\right]
\end{eqnarray*}
as $w_{D},w_{F} \to 0$.\

In particular, coefficients $g_{i,\overline{j}}$ for $1 \leq i,j \leq n-2$ come from {\kah} metrics $\omega_{0}$ and $\gamma_{v}^{\beta}$.\
Thus,
\begin{eqnarray*}
\left[
\begin{array}{ccc}
g_{1,\overline{1}} & \cdots & g_{1,\overline{n-2}} \\
\vdots & \ddots & \vdots \\
g_{n-2,\overline{1}}& \cdots &g_{n-2,\overline{n-2}}
\end{array}
\right]
=
O(|w_{D}|^{-2\hat{S}_{D}/n(n-1)} + (|w_{F}|^{2 \beta} + v)^{-1/\beta}).
\end{eqnarray*}
For other blocks, we similarly have
\begin{eqnarray*}
\left[
\begin{array}{ccc}
g_{1,\overline{n-1}} & g_{1,\overline{n}} \\
\vdots & \vdots \\
g_{n-2,\overline{n-1}} & g_{n-2,\overline{n}}
\end{array}
\right]
=
O(|w_{D}|^{-2\hat{S}_{D}/n(n-1)} + (|w_{F}|^{2 \beta} + v)^{-1/\beta}).
\end{eqnarray*}

From Lemma 3.4 in \cite{Aoi1}, we have
\begin{eqnarray*}
g^{i,\overline{j}}
=
\left[ 
\begin{array}{ccccc}
g^{1,\overline{1}} & \cdots & g^{1,\overline{n-2}} & g^{1,\overline{n-1}} & g^{1,\overline{n}} \\
\vdots & \ddots & \vdots& \vdots & \vdots \\
g^{n-2,\overline{1}}& \cdots &g^{n-2,\overline{n-2}}&g^{n-2,\overline{n-1}}&g^{n-2,\overline{n}}\\
g^{n-1,\overline{1}} & \cdots & g^{n-1,\overline{n-2}} & c(|w_{F}|^{2 \beta} + v)^{2/\beta} |w_{F}|^{2} &c |w_{D}|^{3+4\hat{S}_{D}/n(n-1)}|w_{F}|\\
g^{n,\overline{1}} & \cdots & g^{n,\overline{n-2}} &c |w_{D}|^{3+4\hat{S}_{D}/n(n-1)}|w_{F}| &c |w_{D}|^{2 + 4\hat{S}_{D}/n(n-1)}
\end{array} 
\right]
\end{eqnarray*}
as $w_{D},w_{F} \to 0$.\
Since metric tensors $g^{i,\overline{j}}$ with $i,j \neq n-1,n$ come from {\kah} metrics $\omega_{0}$ and $\gamma_{v}^{\beta}$ whose scalar curvature have been already known.\
Thus, it is enough to study the case that $i = n-1,n$ and $j= n-1,n$.\
Recall that the components of the Ricci tensor are defined by $R_{i,\overline{j}} := - g^{p,\overline{q}} \partial^{2} g_{p,\overline{q}} / \partial z^{i} \partial \overline{z}^{j} + g^{k,\overline{q}} g^{p,\overline{l}} (\partial g_{k,\overline{l}} / \partial z^{i})( \partial g_{p,\overline{q}} / \partial \overline{z}^{j} )$.\
So, the Ricci form ${\rm Ric }(\omega_{c,v,\eta})$ is written as
\begin{eqnarray*}
\left[ 
\begin{array}{ccccc}
R_{1,\overline{1}} & \cdots & R_{1,\overline{n-2}} & R_{1,\overline{n-1}} & R_{1,\overline{n}} \\
\vdots & \ddots & \vdots& \vdots & \vdots \\
R_{n-2,\overline{1}}& \cdots &R_{n-2,\overline{n-2}}&R_{n-2,\overline{n-1}}&R_{n-2,\overline{n}}\\
R_{n-1,\overline{1}} & \cdots & R_{n-1,\overline{n-2}} & c^{-3} (|w_{F}|^{2 \beta} + v)^{-2/\beta} |w_{F}|^{-2} & c^{-3}|w_{F}|^{-1}|w_{D}|(|w_{F}|^{2 \beta} + v)^{-2/\beta} \\
R_{n,\overline{1}} & \cdots & R_{n,\overline{n-2}} & c^{-3}|w_{F}|^{-1}|w_{D}|(|w_{F}|^{2 \beta} + v)^{-2/\beta} & c^{-3} |w_{D}|^{- 2 - 4 \hat{S}_{D}/n(n-1)}
\end{array} 
\right]
\end{eqnarray*}
as $w_{D},w_{F} \to 0$  and the other components of the Ricci tensor $R_{i,\overline{j}}$ for $1 \leq i \leq n-2$ are under control.\

By taking the trace, we obtain the following:
$$
S(\omega_{c,v,\eta}) = O(c^{-2}).
$$
\end{proof}

\begin{rem}
\label{non asc}
On the region in the previous claim, there are the terms $\partial \Theta(t) \wedge \overline{\partial} \Theta(t)$ and $\partial G_{v}^{\beta}(\beta b) \wedge \overline{\partial} G_{v}^{\beta}(\beta b)$ in the complete {\kah} metric $\omega_{c,v,\eta}$.\
Thus, $(X \setminus D, \omega_{c,v,\eta})$ is not of asymptotically conical geometry and we can't use the analysis in Section 5 of \cite{Aoi1} with respect to this {\kah} metric $\omega_{c,v,\eta}$.\
This problem will be solved in \cite{Aoi3}.\
\end{rem}

We proceed to the estimate of $S(\omega_{c,v,\eta})$ on another region.\
\begin{claim}
Consider the region defined by
\begin{eqnarray*}
\tilde{G}_{v}^{\beta}(b)  + \eta_{2} &<& \max \{\Theta(t) - \eta_{1}, ( t + \varphi +c) - \eta_{3} \},\\
| \Theta(t) - ( t + \varphi +c) | &<& \eta_{1} + \eta_{3}.
\end{eqnarray*}
Take parameters $\eta,\kappa$ so that 
\begin{equation}
\label{parameters}
(1 - \kappa ) c + \kappa \eta_{1} - \eta_{2} = (1 - \kappa+ \kappa a_{1} - a_{2})c = 0
\end{equation}
for any $c>0$.\
Then, we can make the scalar curvature $S(\omega_{c,v})$ small arbitrarily by taking a sufficiently large $c$.\
\end{claim}

\begin{proof}

On this region, since
$$
M_{c,v,\eta} = M_{\eta} \left( \Theta(t), t+\varphi+c \right)
$$
from Lemma \ref{gluing lemma}, we have
\begin{eqnarray*}
\omega_{c,v,\eta}
&=& \frac{\partial M_{c,v,\eta}}{\partial t_{1}} \omega_{0} + \frac{\partial M_{c,v,\eta}}{\partial t_{3}} \dol (t + \varphi)\\
&\hspace{15pt}+& \left[ 
\begin{array}{cc}
\partial \Theta(t) & \partial (t + \varphi)
\end{array} 
\right]
\left[ 
\begin{array}{c}
\frac{\partial^{2} M_{c,v,\eta}}{\partial t_{i} \partial t_{j}}
\end{array} 
\right]
\left[ 
\begin{array}{cc}
\overline{\partial} \Theta(t) & \overline{\partial} (t + \varphi)
\end{array} 
\right]^{t}.
\end{eqnarray*}
From the hypothesis of this claim, we have
\begin{eqnarray*}
G_{v}^{\beta}(\beta b)
&<& (t + \varphi + c) + \eta_{3} - \kappa \Theta(t) - \eta_{2}\\
&<& (1 - \kappa)(t + \varphi + c) + \kappa (\eta_{1} + \eta_{3}) + \eta_{3} - \eta_{2}\\
&=& (1 - \kappa)(t + \varphi) + (1+\kappa)\eta_{3}.
\end{eqnarray*}
By taking a small $v>0$ and a suitable $b_{0}$ in the definition of the function $G_{v}^{\beta}( \beta b)$, we may assume that
$$
\beta b < G_{v}^{\beta}( \beta b).
$$
From a priori estimate due to Ko\l odziej \cite{Ko} again, $\varphi$ is bounded on $X$.\
So, on this region, we have the following inequality:
$$
|| \sigma_{F} ||^{-2\beta/(1-\kappa)} < C || \sigma_{D} ||^{-2}
$$
for some constant $C>0$ depending only on the $C^{0}$-norm of $\varphi$.\
By taking $\kappa$ close to $1$ which depends on $m,l$ and $a = a(n)$ in Theorem \ref{solution}, we may assume that
$$
||\sigma_{F}||^{ -2 -2a/l} < C ||\sigma_{D}||^{-2a m/l}.
$$
Thus, on this region, the growth of derivatives of $\varphi$ can be controlled by the {\kah} metric $\omega_{0}$.\
Take a point in $D \setminus (D \cap F)$ and local holomorphic coordinates $(z^{i})_{i=1}^{n} = (z^{1},...,z^{n-1},w_{D})$ satisfying $D = \{ w_{D} = 0 \}$.\
Then, we have
\begin{eqnarray*}
\left| \frac{\partial^{2}}{\partial z^{i} \partial \overline{z}^{j}} \partial^{\alpha} \varphi \right| = O\left( |w_{D}|^{-2a m/l} \right),
\end{eqnarray*}
if $1 \leq i,j \leq n-1$ and
\begin{eqnarray*}
\left| \frac{\partial^{2}}{\partial w_{D} \partial \overline{w_{D}}} \partial^{\alpha} \varphi \right| = O\left( |w_{D}|^{ -2 -2am/l} \right).
\end{eqnarray*}

Similarly, we have
\begin{lemma}
\label{quadratic terms}
By taking a suitable local holomorphic trivialization of $L_{X}$, we may assume that if $(z^{1},...,z^{n-1},w_{D}) = (0,...,0,w_{D})$, we have
\begin{eqnarray*}
\partial \Theta(t) \wedge \overline{\partial} \Theta(t)
&=& O(|w_{D}|^{- 2 - 4\hat{S}_{D}/n(n-1)}) dw_{D} \wedge d \overline{w_{D}}.\\
\end{eqnarray*}
\end{lemma}

Recall the hypothesis
$$
\frac{am}{2l} < \frac{\hat{S}_{D}}{n(n-1)}.
$$
So, Theorem \ref{solution} implies that the growth of the {\kah} metric $\omega_{c,v,\eta}$ is greater then the growth of the higher order derivatives of $\varphi$.\
Thus, Lemma 3.4 in \cite{Aoi1} shows that higher order derivatives including $\partial^{4} \varphi /\partial w^{2} \partial \overline{w}^{2}$ are controlled by taking the trace with respect to $\omega_{c,v,\eta}$.\
Therefore, we can ignore derivatives of $\varphi$ arising in the components of the Ricci tensor and we have
$$
S(\omega_{c,v,\eta}) = O(c^{-2}).
$$
\end{proof}

We proceed to the estimate of $S(\omega_{c,v,\eta})$ the following region.
\begin{claim}
Consider the region defined by
\begin{eqnarray*}
 \Theta(t) + \eta_{1} &<& \max \{\tilde{G}_{v}^{\beta}(b) - \eta_{2}, ( t + \varphi +c) - \eta_{3} \},\\
| \tilde{G}_{v}^{\beta}(b) - ( t + \varphi +c) | &<& \eta_{2} + \eta_{3}.
\end{eqnarray*}
By choosing sufficiently small number $v>0$ so that
$$
(||\sigma_{F}||^{2\beta} + v)^{2/\beta} < ||\sigma_{F}||^{4am/l}
$$
holds on this region, we can make the scalar curvature $S(\omega_{c,v})$ small arbitrarily by taking a sufficiently large $c$.\
\end{claim}

\begin{proof}
The reason why we can find a sufficiently small number $v>0$ satisfying the statement in this claim is that $\min \{ ||\sigma_{F}|| \}$ on this region increase as $v \to 0$ and $4am/l < 4$.\
In order to prove this Claim, we need the following lemma.
\begin{lemma}
\label{quadratic terms}
By taking a suitable local trivialization of $K_{X}^{-l} \otimes L_{X}^{m}$, we may assume that if $(z^{1},...,z^{n-2},w_{F},z_{n}) = (0,...,0,w_{F},0)$, we have
\begin{eqnarray*}
\partial G_{v}^{\beta}(\beta b) \wedge \overline{\partial} G_{v}^{\beta}(\beta b)
&=& O((|w_{F}|^{2 \beta} + v)^{-2/\beta} |w_{F}|^{-2}) dw_{F} \wedge d \overline{w_{F}}.
\end{eqnarray*}
\end{lemma}
Thus, we can prove this claim by using the same way in the previous claim.
\end{proof}

The remained case is the following claim.\
\begin{claim}
On the region defined by
\begin{eqnarray*}
| \Theta(t) -  \tilde{G}_{v}^{\beta}(b) | &<& \eta_{1} + \eta_{2},\\
| \tilde{G}_{v}^{\beta}(b) - ( t + \varphi +c) | &<& \eta_{2} + \eta_{3},\\
| \Theta(t) - ( t + \varphi +c) | &<& \eta_{1} + \eta_{3},
\end{eqnarray*}
we can make the scalar curvature $S(\omega_{c,v,\eta})$ small arbitrarily by taking a sufficiently large $c$.\
\end{claim}

\begin{proof}
On this region, we can show that $S(\omega_{c,v,\eta}) = O(c^{-2})$ similarly.\
Thus, we have finished proving Theorem \ref{small scalar curvature}.\
\end{proof}

%\begin{acknowledgements}
%\end{acknowledgements}




\bigskip
\address{
(CURRENT ADDRESS)\\
Osaka Prefectural Abuno High School,\\
3-38-1, Himuro-chou, Takatsuki-shi,\\
Osaka, 569-1141\\
Japan
}
{takahiro.aoi.math@gmail.com}

\bigskip
\address{
(OLD ADDRESS)\\
Department of Mathematics\\
Graduate School of Science\\
Osaka University\\
Toyonaka 560-0043\\
Japan
}
{t-aoi@cr.math.sci.osaka-u.ac.jp}

\end{document}